\documentclass{amsart} 
\usepackage{graphicx}

\usepackage{amsmath,amsfonts,amsthm,amssymb,graphics,color, amsrefs}

\newtheorem{theorem}{Theorem}
\newtheorem{prop}{Proposition}

\theoremstyle{remark}
\newtheorem{remark}{Remark}
\theoremstyle{definition}
\newtheorem{defn}[equation]{Definition}

\numberwithin{equation}{section}

\newcommand{\pa}{\partial}
\newcommand{\eps}{\varepsilon}

\newcommand{\bx}{{\bf x}}

\newcommand{\cN}{{\mathcal{N}}}

\newcommand{\cA}{{\mathcal{A}}}

\newcommand{\cH}{{\mathcal{H}}}
\newcommand{\N}{\mathbb{N}}
\newcommand{\R}{\mathbb{R}}
\newcommand{\cV}{\mathcal{V}}
\newcommand{\C}{\mathbb{C}}

\newcommand{\bS}{\mathbb{S}}

\newcommand{\cC}{\mathcal{C}}
\newcommand{\cL}{\mathcal{L}}
\newcommand{\x}{\textbf{x}}
\newcommand{\y}{\textbf{y}}

\newcommand{\os}{\overline{s}}
\newcommand{\us}{\underline{s}}

        \definecolor{pink}{rgb}{1,0,1}
	
        \definecolor{purple}{rgb}{0.4,0.2,1}

\begin{document}
\title[Conformal deformations of conic metrics to csc]{Conformal deformations of conic metrics to constant scalar curvature}

\author[Thalia Jeffres]{Thalia Jeffres}
\address{Department of Mathematics \\ Wichita State University \\ 1845 N Fairmount, Wichita, KS 67260-0033.} \email{jeffres@math.wichita.edu}

\author[Julie Rowlett]{Julie Rowlett}
\address{Hausdorff Center for Mathematics\\ Villa Maria
Endenicher Allee 62 \\ D-53115 Bonn \\ Germany.} 
\curraddr{Chalmers University and the University of Gothenburg \\ Mathematical Sciences \\ 41296 Gothenburg Sweden} 
\email{julie.rowlett@chalmers.se} 


\maketitle

\begin{abstract}

We consider conformal deformations within a class of incomplete Riemannian metrics which generalize conic orbifold singularities by allowing both warping and any compact manifold (not just quotients of the sphere) to be the ``link'' of the singular set.  Within this class of ``conic metrics,'' we determine obstructions to the existence of conformal deformations to constant scalar curvature of any sign (positive, negative, or zero).  For conic metrics with negative scalar curvature,  we determine sufficient conditions for the existence of a conformal deformation to a conic metric with constant scalar curvature $-1$; moreover, we show that this metric is unique within its conformal class of conic metrics.    Our work is in dimensions three and higher.       
\end{abstract}

\section{Introduction}


Among all Riemannian metrics, mathematicians and physicists have particular interest in certain canonical metrics.  A natural questions is, \em given a Riemannian metric on a smooth manifold, does there exist a conformal deformation to a metric with constant scalar curvature?  \em  This is known as the Yamabe problem and has been completely solved for closed manifolds through a series of papers starting with Yamabe's work \cite{ya}, followed by  \cite{tru} and \cite{aub}, and completed by \cite{scho}; an excellent survey is \cite{lp}.  After the solution of the Yamabe problem on closed manifolds, it is natural to consider the problem for open manifolds.  The first difficulty is to determine an appropriate formulation.  When the open manifold is the complement of a submanifold $\Sigma$ of a closed Riemannian manifold $(M,g)$, the ``singular Yamabe problem'' is to find a \em complete \em metric conformal to $g$ on $M - \Sigma$ which has constant scalar curvature.  This problem has been studied by several authors including \cite{f1}, \cite{fm}, \cite{ms}, and \cite{mc}.  Solutions to semilinear elliptic equations (including the Yamabe equation) in this setting have been investigated in \cite{f3}, \cite{f4}, \cite{f5}, and \cite{fm}.   In a similar vein, given a complete non-compact Riemannian manifold, \cite{amco} and \cite{f2} study conformal deformations to complete metrics with constant scalar curvature.  We are interested in a related, but different, singular Yamabe problem where the initial metric is incomplete with a particular singular structure at the boundary, and this structure must be preserved by the conformal deformation.  Our problem is:  \em given an open manifold with conic metric, does there exist a conformal deformation to a conic metric with constant scalar curvature?  \em  Akutagawa-Botvinnik's work concerning the Yamabe problem on manifolds with cylindrical ends \cite{ab2} and the Yamabe invariant of cylindrical manifolds and orbifolds \cite{ab1}, \cite{ab3} is closely related to this problem.  

The metrics we consider are smooth Riemannian metrics on the interior of a manifold with boundary which have a particular degenerate structure at the boundary; the precise definitions are contained in the following section.  For a smooth topological manifold $X$ with boundary $\pa X^m$ and $X \cup \pa X$ compact, with respect to a boundary defining function $x$ and local coordinates $(x, \y)$ in a collar neighborhood of $\pa X$, a conic metric $g$ has the degenerate form
$$g = dx^2 + x^2 \sum_{i,j=1} ^m a_{ij} dy_i dy_j = dx^2 + x^2 h(x, \y; d\y).$$
For the conic metrics we consider, $h(0, \y; d\y)$ is a metric on $\pa X$ and is independent of the choice of $x$.  Our first result concerns the obstructions.  Note that throughout this work, the dimension of the boundary is $m$, and the total dimension is $m+1$.   

\begin{prop}  Let $X$ be a smooth topological manifold with boundary $\pa X = Y^m$, $m \geq 2$, $X \cup \pa X$ compact, and conic metric $g$.  If there exists a conic metric which is conformal to $g$ and has constant scalar curvature, then $g$ satisfies the following.
$$ \begin{array}{ll}
1. & \textrm{The scalar curvature of $(\pa X, h(0, \y; d\y))$ is a positive constant.} \\ 
2. & \sum_{k,l=1} ^m \left( h|_{x=0} \right)^{kl}\frac{\partial (h|_{x=0})_{kl}}{\partial x} =0. \end{array} $$
\end{prop}

In our main result, we construct conformal deformations to constant negative scalar curvature.  Note that the hypothesis allows the scalar curvature to vanish at $\pa X$.  

\begin{theorem}  Let $X$ be a smooth topological manifold with boundary $\pa X = Y^m$, $m \geq 2$, $X \cup \pa X$ compact, and conic metric $g$ with negative scalar curvature on $X$.  Assume the scalar curvature of $(Y, h(0, \y; d\y)) \equiv m(m-1)$, and $h$ satisfies condition $2$ of Proposition 1.   Then, there exists a unique conic metric $\tilde{g}$ which is conformal to $g$ and has constant negative scalar curvature $-1$.  
\end{theorem}

This work is organized as follows:  \S 2 contains geometric definitions, curvature calculations, and preliminary geometric results.  Analytic preliminaries, proof of the obstructions, and the geometric interpretations of the obstructions comprise \S 3.  Our main theorem is then proven in \S 4.  A brief discussion of positive and negative examples and the orbifold Yamabe invariant comprise \S 5.

\section{Geometric preliminaries}
The conic metrics we consider are defined in terms of Melrose's ``b-metrics'' \cite{tapsit}.  
\begin{defn} \label{def:bdf} \em 
Let $X$ be a smooth topological manifold with boundary $\pa X$ and smooth structure on $X \cup \pa X$.  A boundary defining function $x:X \to [0, \infty)$ is a smooth function on $X$ up to $\pa X$ such that $x^{-1}( \{0\}) = \pa X$ and $dx \neq 0$ on $\pa X$.    \em \end{defn} 

Recall from \cite{tapsit} the space $\cV_b$ of b-vector fields over $X$ defines a Lie algebra and is a finitely generated $\cC^{\infty}(X)$-module.  In terms of the local coordinates $(x, \y)$ in a collar neighborhood of $\pa X$ diffeomorphic to $(0, x_0)_x \times \pa X_{\y}$, the vector fields 
$$\{x \pa_x, \pa_{y^1}, \ldots, \pa_{y^m}\}$$
are linearly independent elements of $\cV_b$ which locally span over $\cC^{\infty}(X).$  There is a bundle $^b TX$ naturally associated to $\cV_b$ such that $\cV_b = \cC^{\infty}(X, \, ^b TX).$  The dual bundle $^b T^*X$ is locally spanned by the 1-forms 
$$\left\{ \frac{dx}{x}, dy_1, \ldots, dy_m \right\}.$$
A \em b-metric \em is a metric on the fibers of $^b TX$; on the interior, a b-metric is a Riemannian metric, and at the boundary it has a certain structure.  With respect to local coordinates at the boundary it can be written as 
\begin{equation} \label{eq:bdef} g = a_{00} \left( \frac{dx}{x} \right)^2 + 2 \sum_{j=1} ^{m} a_{0j} \frac{dx}{x} dy^j + \sum_{j,k=1} ^{m} a_{j,k} dy^j dy^k, \end{equation}
where the coefficients are smooth on $X$, and the form is positive definite
\begin{equation} \label{eq:posdef} a_{00}\lambda^2 + 2 \sum_{j=1} ^{m} a_{0j} \lambda \eta_j + \sum_{j,k=1} ^{m} a_{j,k} \eta_j \eta_k \geq \epsilon ( |\lambda|^2 + |\eta|^2) \, \forall \, \lambda \in \R, \, \eta \in \R^{m}. \end{equation} 
As observed in \cite{tapsit}, the element $x \pa_x \in T_p X$ is well defined at $p \in \pa X$, so the condition $a_{00} |_{\pa X} = $ constant is well defined.  We shall always make this assumption; moreover, since one can rescale the metric to make this constant equal to one, we will assume in this work that $a_{00} |_{\pa X} =1$.  

A particularly nice class of b-metrics are the exact b-metrics, which we recall from \cite{tapsit}.
\begin{defn} \label{def:exact} \em An \em exact b-metric \em on a compact manifold with boundary is a b-metric such that for some boundary defining function $x$, 
$$g = \left( \frac{dx}{x} \right)^2 + g', \quad g' \in \cC^{\infty} (X; T^* X \otimes T^* X).$$ \em 
\end{defn}
An exact b-metric is simply a b-metric such that, for some choice of $x$, $a_{00} = 1 + O(x^2)$ and $a_{0j} = O(x)$ in (\ref{eq:bdef}).  

\begin{defn} \label{def:cone} \em Let $X$ be a smooth topological manifold with boundary $\pa X = Y^m$ and $X \cup \pa X$ compact.  A \em conic metric, \em is a Riemannian metric $g$ on the interior of $X$ such that, with respect to a boundary defining function $x$ and local coordinates $(x, \y)$ on a non-empty neighborhood $\cN(\pa X) \cong [0,x_0)_x \times Y_{\y}$, 
$$x^{-2} g \textrm{ is an exact b-metric with $a_{00}|_{\pa X} \equiv 1$.}$$  \em 
\end{defn} 

\begin{remark}
As discussed in \cite{tapsit}, for a conic metric, there exists an $x$ with 
$$g = dx^2 + x^2 h(x, \y; d\y)$$
in a neighborhood of $\pa X$.  Furthermore $h(0, \y; d\y)$ is a Riemannian metric on $\pa X$ and is  well defined independent of the choice of $x$.  For this reason, our work focuses on (exact) conic metrics, which we simply call \em conic metrics; \em they are also the incomplete analogues of the cylindrical metrics studied in \cite{ab1}, \cite{ab2}, \cite{ab3}.  If the metric $h$ is \em independent of $x$ \em in a neighborhood of $\pa X$, then $g$ is a \em rigid conic metric.  \em  
\end{remark}  

\subsection{Curvature calculations}  
Recall the Christoffel symbols 
$$\Gamma_{ji} ^k = \Gamma_{ij} ^k := \frac{1}{2} \sum_{l} g^{kl} \left( \frac{\partial g_{il}}{\partial x_j} + \frac{\partial g_{jl}}{\partial x_i} - \frac{\partial g_{ij}}{\partial x_l} \right),$$
where the metric $g$ as a matrix has entries $g_{ij},$ and its inverse $g^{-1}$ has entries $g^{ij}.$  The Riemannian curvature tensor in coordinates is 
$$R^{l} _{ijk} = \frac{ \partial \Gamma^l _{ik}} {\partial x_j} - \frac{\partial \Gamma^l _{ij}}{\partial x_k} + \sum_{m} ( \Gamma^m _{ik} \Gamma^l _{mj} - \Gamma^m _{ij} \Gamma^l _{mk}),  $$
and in particular
\begin{equation} \label{curvw} R^{k} _{ikj} = \frac{ \partial \Gamma^k _{ij}} {\partial x_k} - \frac{\partial \Gamma^k _{ik}}{\partial x_j} + \sum_{m} ( \Gamma^m _{ij} \Gamma^k _{mk} - \Gamma^m _{ik} \Gamma^k_{mj}). \end{equation}
The scalar curvature is the full contraction of the Riemannian curvature tensor:  
\begin{equation} \label{scurv} S= \sum_{i,j,k} g^{ij} R^k _{ikj}.  \end{equation}
Straightforward calculations yield the following.  

\subsubsection{Singular terms in the scalar curvature of a conic metric}
The singular terms in the scalar curvature of a conic metric are precisely, 
\begin{equation} \label{eq:edgesing} x^{-2} \left[ S(h, y) - (m)(m-1)- mx \left . \left( \sum_{k,l=1} ^m h^{kl} \frac{\partial (h)_{kl}}{\partial x} \right) \right|_{x=0} \right]. \end{equation}
An immediate consequence of this calculation is the following result of boundedness.  

\begin{prop}\label{prop:sbd}  Let $g$ be a conic metric on a smooth topological manifold with boundary $X \cup \pa X^m$.  Then, $S(g)$ is bounded in a neighborhood of $\pa X$ if and only if $S(\pa X, h(0, \y; d\y) ) \equiv m(m-1)$, and $h$ satisfies condition 2 of Proposition 1.    
\end{prop}

For a rigid conic metric,  $h$ is independent of $x$ in a neighborhood of $\pa X$, so (\ref{eq:edgesing}) becomes
\begin{equation}\label{eq:exactsing}  x^{-2} \left[ S(h) -(m)(m-1) \right]. \end{equation}
An immediate consequence of this calculation is the following. 

\begin{prop}  There is no Riemannian manifold with boundary with constant negative or positive scalar curvature and rigid conic metric. \end{prop}

\begin{proof} Since the calculation (\ref{eq:exactsing}) is valid on in a neighborhood of $\pa X = \{ x=0\}$, either the scalar curvature is identically zero on this neighborhood or the scalar curvature approaches $\pm \infty$, depending on the value of $S(h(\y))$.  
\end{proof}

\subsection{Scalar curvature for a general warped product metric}
Let $X \cup \pa X^m$ be a smooth topological manifold with boundary and Riemannian metric $g$ on $X$.  Let $h$ be a Riemannian metric on $\pa X$.  Assume there exists a boundary defining function $x$, a collar neighborhood of the boundary $\cN(\pa X) \cong (0, x_0)_x \times Y_\y$, and a smooth positive function $f^2(x)$ on $X$ which depends only on $x$ near $\pa X$ such that 
$$g = dx^2 + f^2(x) h(\y, d\y), \qquad \textrm{ on } \cN(\pa X).$$
As computed in \cite[2.2]{leung} the scalar curvature of $g$ on $\cN (\pa X)$ is 
\begin{equation} \label{eq:wp}S(g) =  \frac{1}{f^2(x)} \left[ S(h)(\y) - 2 m f(x) f''(x) - m(m-1)(f'(x))^2 \right],\end{equation}
where $S(h)(\y)$ is the scalar curvature of $Y$ at the point $\y$ with respect to the metric $h$.
This calculation is useful for constructing not only conic metrics (see \S 5), but also more general warped product metrics.

\section{Analytic preliminaries}
We recall some basic facts and definitions about Melrose's ``b-operators'' and ``b-calculus'' \cite{tapsit}, \cite{mm}, \cite{bb}. 

A \em differential b-operator \em is a differential operator which may be expressed locally as finite sums of products of elements of $\cV_b$, and we write
\begin{equation} \label{eq:diffedge} L \in \textrm{Diff}_b^* (X) \iff L = \sum_{j + |\alpha| \leq N} a_{j, \alpha} (x, \y) (x\pa_x)^j (\pa_{\y})^{\alpha}.\end{equation} 
An element of Diff$_b ^*(X)$ is \em elliptic \em if it is expressed as an elliptic linear combination of the vector fields in (\ref{eq:diffedge}).  Alternatively, one can define a filtration of Diff$_b ^* (X)$ by subspaces of operators of order at most $N$, together with a principle symbol map acting on $^b T^* X$ and a corresponding invariantly defined homogeneous polynomial of degree $N$, and then $L$ is said to be elliptic if this symbol does not vanish off the zero section.  The Laplace-Beltrami operator $\Delta$ for a conic metric $g$ is a weighted, or renormalized, elliptic differential b-operator:  
$$ \textrm{there exists $\Delta_b \in $ Diff$_b ^* (X)$ such that $\Delta = x^{-2} \Delta_b$.}$$

Recall from \cite{tapsit}, the basic conormal space of functions, 
$$\cA^0 (X) = \{ u : V_1 \ldots V_l u \in \cL^{\infty} (X), \quad \forall \, V_i \in \cV_b, \textrm{ and $\forall$ } l\}.$$
The space of polyhomogeneous conormal functions $\cA_{phg} ^* (X) $ consists of all conormal functions admitting an asymptotic expansion of the following form
$$u\in \cA_{phg} ^* (X): \,  u \sim \sum_{\mathfrak{R} s_j \to \infty} \sum_{p=0} ^{p_j} x^{ s_j} (\log x)^p a_{j,p} (y,z), \, \, a_{j,p} \in \cC^{\infty}(X \cup \pa X),$$
where the exponents $\{s_j\} \in \C$.  These expansions generalize classical Taylor expansions to accommodate manifolds with boundary  \cite{bb}.   The weighted H\"older and Sobolev spaces have the most convenient mapping properties for the Laplace operator $\Delta$ associated to conic metrics.  
\begin{defn} \em When $l \in \N_0 := \N \cup \{0\},$ let
$$x^{\delta} H^l _b (X; \Omega^{\frac{1}{2}}) = \{ u = x^{\delta} v; V_1 \ldots V_j v \in \cL^2(X; \Omega^{\frac{1}{2}}), \, \, \forall j \leq l, \, \, V_i \in \cV_b \}, \quad H_b ^{\infty} := \bigcap_{l \in \N} H_b ^l.$$
Above, $\Omega^{\frac{1}{2}}$ is the standard half density bundle over $X$; in terms of the local coordinates $(x,\y)$ near $\pa X$ a canonical nonvanishing half-density $\mu = \sqrt{dx d\y}$.  
\em \end{defn}

Given any differential b-operator $L$ as in (\ref{eq:diffedge}), its indicial operator is defined to be
$$I(L) := \sum_{j + |\beta| \leq N } a_{j, \beta} (0, y) (s \pa_s )^j (\pa_y)^{\beta}.$$
The indicial operator is conjugated to $I_{i \zeta} (L)$ by the Mellin transform which results in an indicial family of operators depending on the complex parameter $\zeta.$  
\begin{defn} \em If $L \in $ Diff$_b^* (X)$ is elliptic, the boundary spectrum of $L$, $spec_b(L)$, is the set of values $\zeta \in \C$ for which the operator $I_{\zeta} (L)$ fails to be invertible on $\cL^2 (Y).$ \em \end{defn}
For the b-operator $\Delta_b$ such that $\Delta = x^{-2} \Delta_b$, $spec_b (\Delta_b)$ is a discrete subset of $\C$ \cite{tapsit}.  
\subsection{The Yamabe equation} 
The following abbreviated derivation of the Yamabe equation applies to a general Riemannian manifold \( (M,g) \) of dimension $m+1$.  We take the Laplacian to be the divergence of the gradient, so that on $\R^{m+1} $ with the Euclidean metric 
\( \Delta  = \sum _{i=1} ^{m+1} \partial_{x_i} ^{2} ,\) and on a Riemannian manifold \( (M,g) ,\) 
\begin{equation} \label{eq:laplace} \Delta = \sum_{j,k=1}^{m+1} \frac{1}{\sqrt{\det(g)}} \pa_j \left(g^{jk} \sqrt{\det(g)} \right) \pa_k . \end{equation} 
If \( \tilde{g} = e^{f} g ,\) then the scalar curvature of \( \tilde{g} \) is related to that of $g$ by 
\[ e^{f} S(\tilde{g} ) = S(g) -(m) \Delta f -\frac{1}{4} (m)(m-1) <\nabla f,\nabla f > .\] 
Writing instead \( \tilde{g} = u^{c} g \) and making the right choice of $c$ eliminates the gradient 
term; the right choice proves to be \( c = 4/(m-1) .\) Reorganizing terms, 
\[ \Delta _{g} u - \frac{(m-1)}{4m} S(g) u + \frac{(m-1)}{4m} S(\tilde{g} ) u^{(m+3)/(m-1)} = 0.\] 
If $S(\tilde{g})$ is prescribed in advance, and if this equation has a positive solution $u$, then 
\( \tilde{g} = u^{4/(m-1)} g \) has scalar curvature equal to $S(\tilde{g})$. We wish to solve this equation for $S(\tilde{g}) = -1$, so we define
\begin{equation} \label{eq:y} Pu:= \Delta u - a S u - a u^{\frac{m+3}{m-1}},\end{equation}
where $a = \frac{m-1}{4m}$, and $S$ is the scalar curvature with respect to the metric $g$.  Then, 
$$S(\tilde{g}) = -1 \iff Pu = 0.$$

Working \em within the conformal class of conic metrics \em places restrictions on the conformal deformation $u$.  


\begin{prop} \label{prop:cds} Let $g$ be a conic metric on a smooth topological manifold with boundary $X \cup \pa X$, with dimension of $\pa X = m \geq 2$.  Let $u$ be a smooth positive function on $X$ with $u \in \cA_{phg} ^*$.  Then, $u g$ is a conic metric if and only if the expansion of $u$ at $\pa X$ satisfies
\begin{equation} \label{eq:expu} u \sim u_0 + O(x^{2}), \quad u_0 \equiv \textrm{ positive constant.} \end{equation} 
\end{prop}
\begin{proof}
If $u$ is of this form, then with respect to the boundary defining function $t:= \sqrt{u_0} x$, the metric $ug$ is a conic metric.  Conversely, if $ug$ is a conic metric, there exist local coordinates $(t, \y)$ in a neighborhood of $\pa X$ such that $g = dt^2 + t^2 h(t, \y; d\y)$.  This means that $t$ is also a boundary defining, so in terms of $x$, $t\sim cx$ for a constant $c>0$ as $x \to 0$.  In order to have the metric of this (conic) form, it follows that $t=cx+O(x^2)$ as $x \to 0$.
\end{proof}

\subsection{The obstructions}  
The boundary determines the existence of conformal deformations to constant scalar curvature \em of any sign.  \em  Our obstructions are related to Convention 1 and Remark 3.1 of \cite{ab3} for orbifolds with conic singularity.  

\begin{proof}
By the preceding proposition, the only polyhomogeneous conormal deformations which preserve exact conic metrics are those which have an expansion near $\pa X$
$$u \sim u_0 x^{\alpha} + O(x^{\alpha + 2}) \quad \alpha > -2.$$
If the original metric 
$$g = dx^2 + x^2 h(x, \y; d\y),$$
then the metric
$$ug = dt^2 + t^2 \tilde{h}(t, \y; d\y),$$
with
\begin{equation}\label{eq:change} t = \frac{2 \sqrt{u_0} x^{\frac{\alpha +2}{2}}}{\alpha + 2} + O(x^{\frac{\alpha + 6}{2}}) \quad \tilde{h}(0, \y; d\y) = \frac{4}{(\alpha +2)^2} h(0, \y; d\y). \end{equation}
If $g$ is conformal to a conic metric $ug$ with constant scalar curvature, by Proposition \ref{prop:sbd}, 
$$S(\tilde{h}_0) \equiv m(m-1),$$
and
\begin{equation} \label{eq:h0} \sum_{i,j = 1} ^m \left( h_0 \right)^{ij} \frac{\pa (h_0)_{ij}}{\pa x} = 0. \end{equation} 
Since 
\begin{equation} \label{eq:ug1} S(\tilde{h}_0) = \frac{(\alpha +2)^2}{4} S(h_0), \end{equation}
this implies $S(h_0) \equiv c > 0$.  Since $\tilde{h}_0$ is just a scaling of $h_0$, (\ref{eq:h0}) is satisfied for $\tilde{h_0}$ if and only if it is satisfied for $h_0$.  
\end{proof}
\subsection{Geometric interpretation of the obstructions} 
For a rigid conic metric, if the scalar curvature of $(Y, h)$ is identically 1, scaling $h$ by $\frac{1}{m(m-1)}$ gives $(Y, h)$ constant scalar curvature $m(m-1)$.  This is equivalent to 
$$g \mapsto dx^2 + \frac{1}{m(m-1)} x^2 h(x,\y; d\y).$$
Geometrically, this means that if $(Y, h|_{x=0})$ has constant scalar curvature $1$, then the cone angle over $Y$ must be $\sqrt{\frac{1}{m(m-1)}}$.    

To give a geometric interpretation of the second obstruction, recall the second fundamental form operator $B(W,Z) = (\overline{\nabla_{\overline{W}}} \overline{Z} ) ^{\perp}$, where $W$ and $Z$ are vector fields tangent to $Y$ and $\overline{W}$, $\overline{Z}$ are their extensions to $X$ is a symmetric tensor which depends only on the vectors at $p$ and is independent of the choice of extensions.  In a neighborhood of $\pa X$, $\pa_x$ is a unit vector field normal to the submanifold $Y$ at $\{ x=\textrm{ constant} \}$, and $\{ \pa_{y^i} \} _{i=1} ^m $ is a basis for $T_p Y$.  The map $H: T_p X \to T_p X$ is characterized uniquely by 
$$\langle H(W), Z \rangle = \langle B(W, Z), \pa_x \rangle.$$
The mean curvature of the submanifold $Y$ is the trace of $H$ (or a scalar multiple of the trace of $H$, depending on normalization).  We compute that the trace of $H$ is 
$$- \frac{m}{x} - \frac{1}{2} \sum_{k,l = 1} ^m h^{kl} \frac{\pa h_{kl}}{\pa x}.$$
By this calculation, the mean curvature of a rigid conic metric is 
$$-\frac{m}{x},$$
since $h$ is independent of $x$.  Then, the obstruction may be reformulated to:  with respect to $x$, the mean curvature of the conic metric must be asymptotic to the mean curvature of a rigid conic metric.  

For example, if one considers an artificial conical singularity, such as the origin in $\R^{m+1}$ with polar coordinates $(r, \theta)$, then the link is the sphere which has constant scalar curvature $m(m-1)$, and the mean curvature with respect to $r$ is $-mr^{-1}$.  

\section{Proof of Theorem 1}
If the initial metric has negative scalar curvature, the existence of a conformal deformation to constant negative scalar curvature follows from arguments in \cite{f3}, \cite{f4}, \cite{f5}, \cite{fm}.  However, these results do not specifically address our problem to produce a \em conic metric.  \em  We solve the Yamabe equation constructively, simultaneously demonstrating that a solution exists and results in a unique conic metric.  

\subsection{Uniqueness}
\begin{proof}
Let $g$ be a conic metric on $X$ satisfying the hypotheses of the theorem.  For smooth positive polyhomogeneous conormal functions $u$, $v$, assume $\tilde{g} = ug$ and $g^* = vg$ are conic metrics with constant scalar curvature.  Then, it follows from Proposition \ref{prop:sbd} that for the corresponding $\tilde{h}, h^*$, 
$$S(\tilde{h}) = S(h^*) = S(h) \equiv m(m-1).$$
By the calculation (\ref{eq:ug1}), if 
$$u \sim O(x^{\alpha}) \textrm{ as } x \to 0,$$
then
$$S(\tilde{h}) = \frac{(\alpha +2)^2}{4} S(h) \implies \alpha = 0.$$
The same argument applies to $v$.  Scaling the metrics $\tilde{g}$ and $g^*$ appropriately, we may assume $u, v \sim 1 + O(x^2)$ as $x \to 0$.  The maximum principle then implies $u \equiv v$, and the metrics $\tilde{g}$ and $g^*$ are simply related by scaling.  Since $S(\tilde{g}) = S(g^*) = -1$, the scaling factor must be $1$ so in fact $\tilde{g} = g^*$.  
\end{proof} 

\subsection{Existence}  
\begin{proof}
Recall that for a partial differential operator $P$, a supersolution and a subsolution are smooth functions $\psi$, $\phi$, respectively, which satisfy
\begin{equation} \label{eq:barriers} P\psi \leq 0, \qquad P \phi \geq 0, \qquad \phi \leq \psi.  \end{equation} 
Together, these functions form a ``barrier'' within which one constructs a solution to the homogeneous equation $Pu =0$.  We first construct smooth ``barrier functions" $\psi$ and $\phi$ such that 
$$0 < \phi \leq \psi, \qquad P\psi \leq 0 \leq P\phi.$$
We then use an inductive construction to produce a smooth function $v$ which has a polyhomogeneous expansion at $\pa X$ and satisfies 
$$0 < \phi \leq v \leq \psi, \quad Pv = 0, \quad v \sim 1 + O(x^2) \textrm{ at } \pa X.$$

\subsection{Barrier functions} 
Our barrier functions are constant away from $\pa X$ and depend only on $x$ near $\pa X$.  For a conic metric $g$ with 
$$g = dx^2 + x^2 h(x, \y, d\y) \quad \textrm{ on } \quad \cN \cong (0, x_0)_x \times \pa X,$$
its Laplace-Beltrami operator, 
\begin{equation} \label{eq:laplace0} \Delta= \pa_x^2 + \frac{m}{x} \pa_x + \tilde{h} \pa_x + x^{-2} \Delta_h,\end{equation}
with 
 \begin{equation}\label{eq:laplace1} \tilde{h} = \frac{\pa \log \left( \sqrt{ \det(h)}\right)}{\pa x}, \quad \Delta_h = \sum_{j,k=1} ^m \frac{1}{\sqrt{\det(h)}} \pa_j  \left( h^{jk} \sqrt{ \det(h) } \right) \pa_k.  \end{equation} 
Decompose $X$ into $\cN_1 \cup \cN_2$ where 
$$\cN_1 \cong \left( (\epsilon, x_0)_x \times \pa X \right) \cup (X - \cN), \qquad \cN_2 \cong (0, \epsilon)_x \times \pa X.$$
Fix 
\begin{equation} \label{eq:eps0} A := \sup_{(x, \y) \in \cN} \{ 1, |\tilde{h}(x, \y)|, |a x S (x, \y)|\} \quad \textrm{and} \quad \epsilon \leq \inf \left\{ \frac{m}{8A}, \frac{x_0}{2},\frac{A}{m} \right\}.\end{equation}

\subsubsection{Supersolution}
Let 
$$v(x) := \left\{ \begin{array}{ll} b(c - \epsilon)& \textrm{ on } \cN_1 \\ b(c - x) & \textrm{ on } \cN_2 \end{array} \right .$$
where $b$ and $c$ are positive constants chosen to satisfy a set of inequalities.  By standard mollification arguments, there is a smooth function $\psi$ which satisfies
$$v(x) \leq \psi(x) =\left \{ \begin{array}{ll} v(x) & x \leq \epsilon/2\\ 
v(x) & x \geq 3 \epsilon/2 \end{array} \right \},$$
$$-b \leq \psi'  \leq 0, \quad \textrm{and} \quad 0 \leq \psi'' \leq \frac{B b}{\epsilon},$$
for a fixed positive constant $B$ independent of $\epsilon$, $c$, and $b$.  Let
$$\os := \sup_{\{x \geq \epsilon /2 \} } \{ |S|, 1 \}.$$
Since $b(c -  \epsilon) \leq v \leq \psi$, and we require $0 < \psi$, we impose  
\begin{equation} \label{eq:bc0} c > 2 \epsilon. \end{equation}
Considering $P \psi$ away from $\pa X$,  we require 
\begin{equation} \label{eq:bc1} (c - \epsilon) b \geq (\os)^{\frac{m-1}{4}}.  \end{equation}
Since $\psi \geq b(c - \epsilon)$, (\ref{eq:bc1}) implies 
\begin{equation} \label{eq:bc4} \os - \psi^{\frac{4}{m-1}} \leq 0. \end{equation}
For $x < \epsilon / 2$, 
$$P \psi = Pv \leq \frac{ - b m}{x} + A b - a v S - a v^{\frac{m+3}{m-1}} \leq \frac{b}{x} \left( -m + Ax  + Ac \right) - a v^{\frac{m+3}{m-1}}.$$
By (\ref{eq:eps0}), $x \leq \frac{\epsilon}{2} \leq \frac{m}{16A}$, so we require 
\begin{equation} \label{eq:bc2} c \leq \frac{m}{2A}.  \end{equation}
For $\frac{\epsilon}{2} < x < \frac{3\epsilon}{2}$, 
$$P \psi \leq \frac{B b}{\epsilon} + A b - a S \psi - a \psi^{\frac{m+3}{m-1}} \leq b\left( \frac{B}{\epsilon} + A \right) + a\psi \left( \os - \psi^{\frac{4}{m-1}} \right).$$
Since $v \leq \psi$ and (\ref{eq:bc4}), 
$$P \psi \leq b \left( \frac{B}{\epsilon} + A + a(c-\epsilon) (\os - (b(c-\epsilon))^{\frac{4}{m-1}} ) \right),$$
so we impose 
\begin{equation} \label{eq:bc3} b \geq \left( \frac{\frac{B}{\epsilon} + A + a(c-\epsilon)\os }{a(c-\epsilon)^{\frac{m+3}{m-1}}} \right)^{\frac{m-1}{4}}.\end{equation}
We first fix $\epsilon$ to satisfy (\ref{eq:eps0}) which also fixes $\os$.  We then fix $c$ to satisfy (\ref{eq:bc0}), and since $\epsilon < \frac{m}{8A}$, $c$ can be chosen to also satisfy (\ref{eq:bc2}).  Finally, we choose $b > \epsilon^{-1}$ and sufficiently large to satisfy (\ref{eq:bc1}) and (\ref{eq:bc3}). 

\subsubsection{Subsolution}
Using essentially the same idea, let
$$u(x) := \left\{ \begin{array}{ll} c(b+\epsilon^2)  & \textrm{ on } \cN_1 \\ c(b+x\epsilon) & \textrm{ on } \cN_2 \end{array} \right .$$
where $b$ and $c$ are positive constants chosen to satisfy a certain set of inequalities.  By standard mollification arguments, there is a smooth function $\phi$ which satisfies 
$$u(x) \geq \phi(x) =\left \{ \begin{array}{ll} u(x) & x \leq \epsilon/2\\ 
u(x) & x \geq 3 \epsilon/2 \end{array} \right \},$$
$$0 \leq \phi'(x) \leq c\epsilon, \quad \textrm{and} \quad -B c \leq \psi'' \leq 0,$$
for a fixed positive constant $B$ independent of $\epsilon$, $b$ and $c$.  Let
$$\us := \inf_{\{ x \geq \epsilon /2 \} } |S|.$$ 
Assume
\begin{equation} \label{eq:bcc11} c < \frac{1}{b+ \epsilon}, \end{equation}
which implies $c(b+\epsilon) < 1$.  Considering $\phi$ away from $\pa X$, the first condition is 
\begin{equation} \label{eq:bcc0} c \leq  \frac{ ( \us / 2)^{\frac{m-1}{4}} }{b+\epsilon^2}. \end{equation}
Since 
$$\phi \leq u_0 + \epsilon u_1 = cb + c\epsilon^2,$$
(\ref{eq:bcc0}) implies 
\begin{equation} \label{eq:bcc3} \us - \phi^{\frac{4}{m-1}} \geq \us / 2 > 0. \end{equation}
For $x < \epsilon / 2$, since $-S > 0$ and $u(x) = bc + c\epsilon x$,
$$P \phi  = Pu \geq \frac{m c\epsilon}{x} - Ac\epsilon - a S u - a u^{\frac{m+3}{m-1}} \geq 2cm - Ac\epsilon - a \left(bc + \frac{c\epsilon^2}{2} \right)^{\frac{m+3}{m-1}}$$
$$= c\left( 2m-A\epsilon -a c^{\frac{4}{m-1}} \left(b+ \frac{\epsilon^2}{2} \right)^{\frac{m+3}{m-1}} \right).$$
Since $\epsilon < \frac{m}{A}$, we impose 
\begin{equation} \label{eq:bcc1}  c \leq \frac{ m^{\frac{m-1}{4}}}{a^{\frac{m-1}{4}} \left( b+ \frac{\epsilon^2}{2} \right)^{\frac{m+3}{4}}}.  \end{equation}
For $\frac{\epsilon}{2} < x < \frac{3 \epsilon}{2}$, since $bc \leq \phi$ and (\ref{eq:bcc3}),
$$P \phi \geq -B c- Ac\epsilon + abc\left( \us - \phi^{\frac{4}{m-1}} \right) = c \left( -B - A\epsilon + ab( \us - \phi^{\frac{4}{m-1}} ) \right).$$
We may assume \`a priori $\epsilon \leq \frac{B}{A}$ so by (\ref{eq:bcc3}) this becomes 
$$P \phi \geq c \left( -2B + ab (\us - \phi^{\frac{4}{m-1}} ) \right) \geq c \left( - 2B + \frac{ab \us}{2} \right).$$
We impose 
\begin{equation} \label{eq:bcc2} b > \frac{4B}{a \us }.  \end{equation}
To satisfy these conditions, we first fix $\epsilon \leq \frac{B}{A}$ satisfying (\ref{eq:eps0}).  This fixes $\us$, so we may then choose $b$ to satisfy (\ref{eq:bcc2}) and finally choose $c$ to satisfy (\ref{eq:bcc0}) and (\ref{eq:bcc1}).   


\subsection{Construction of the solution} 
Let 
$$f(\x, \varphi) := - aS(\bx)\varphi - a\varphi^{(m+3)/(m-1)}, \qquad \textrm{ for } \x \in X,$$ where $S(\x)$ is the scalar curvature of $(X, g)$.  By construction, there exists a constant $A > 0$ such that 
$$-A \leq \phi \leq \psi \leq A \quad \textrm{ on } X.$$
Let 
$$F(\x, t) := ct + f(\x, t).$$
Then we can choose $c \gg 1$ such that 
\begin{equation} \label{eq:curvneg} \frac{\partial F}{\partial t} > 0 \textrm{ for all } t \in [-A, A].\end{equation} 
Let 
$$L := -\Delta + c.$$
By assumption and Proposition \ref{prop:sbd}, the scalar curvature of $g$ is bounded as $x \to 0$.  Consequently, 
$$F(\x, \phi), \quad F(\x, \psi) \in x^{\epsilon} H_b ^l (X) \cap \cA_{phg}^*,$$
for all $\epsilon < 1/2$, and for all $l$.  Let $L_b$ be the second order elliptic differential b-operator satisfying
$$x^{-2} L_b = L.$$
For $\varphi \in x^{\epsilon} H_b ^l (X)$, 
$$L U = \varphi \iff u:=U-c_1 \textrm{ satisfies } L_b u = x^2 (\varphi - cc_1) \in x^{2+\epsilon} H_b ^l (X).$$
Since $spec_b (\Delta_b)$ is a discrete subset of $\C$ and $spec_b (L_b)$ is a translation of $spec_b (\Delta_b)$, it is also discrete.  Therefore, it is possible to choose $\epsilon < 1/2$ such that $3/2 - \epsilon \neq \mathfrak{R}(\zeta)$ for any $\zeta \in spec_b (L_b)$.  By Theorem (3.8) in \cite{edge}, there exists a parametrix $G$ and remainders $R$ and $R'$ so that 
$$L_b G = I + R, \qquad G L_b = I + R'.$$
Moreover, $G$, $R$, and $R'$ preserve $\cA_{phg} ^*$ and 
$$G: x^{2+\epsilon} H^{l} _b \to x^{2+\epsilon} H^{l+2} _b \textrm{ is bounded,} \quad G: x^{2+\epsilon} H^l _b \to x^{2+\epsilon} H^l _b \textrm{ is compact.}$$  
By Theorem (4.4) in \cite{edge} and our choice of $\epsilon$, $L_b : x^{2+\epsilon} H_b^{l+2} (X) \to x^{2+\epsilon} H_b ^{l} (X)$ is Fredholm for all $l \in \R$.  Therefore, by Theorem (4.20) in \cite{edge}, restricting to $x^{2+\epsilon} H^{l} _b (X)$, the remainder $R$ is projection onto the orthogonal complement of the range of $L_b$.  Given any $\beta \in x^{2+\epsilon} H^{l} _b (X)$, write $\beta = \beta_1 + \beta_2$ so that $\beta_1 \in Im(L_b)$, and $\beta_2 \in Im(L_b)^{\perp}$.  Letting $w = \beta_1 + \frac{\beta_2}{2}$, $w \in x^{2+\epsilon} H^{l} _b (X)$ solves
$$(I + R) w = \beta.$$
For $\beta :=x^2 (\varphi - cc_1)$, $u:= Gw \in x^{2+\epsilon} H^{l+2} _b (X)$ satisfies
$$L_b u = L_b G w = (I + R) w = x^2( \varphi- cc_1) \implies Lu = \varphi - cc_1,$$
and $U := u+c_1$ satisfies
$$LU = \varphi.$$
Apply the preceding construction to $\varphi = F(\bx, \phi), F(\bx, \psi)$, and choose the corresponding constants $c_1$ and $c_1'$ to satisfy
$$c_0 < c_1 < 1 < c_1' < c_0'.$$
Let the corresponding solutions be respectively $\phi_1$ and $\psi_1$, so that $L(\phi_1) = F(\bx, \phi)$ and $L(\psi_1) = F(\bx, \psi)$.  Since $F(\bx, \phi), F(\bx, \psi) \in x^{\epsilon} H_b ^l (X)$ for all $l$, $\phi_1 - c_1$ and $\psi_1 - c_1' \in x^{2+\epsilon} H^{\infty} _b$.  By the Sobolev embedding theorem applied to a compact exhaustion of $X,$ it follows that $\phi_1 - c_1$ and $\psi_1 - c_1  \in \cC^{\infty} (X)$.  Since $\phi, \psi \in \cA_{phg} ^*$ and $g$ is a conic metric, by the mapping properties of $G$ and $R$, $\phi_1, \psi_1 \in \cA_{phg}^*$.  Since $\phi_1 \to c_1$ at $\pa X$ and lies in $\cA_{phg} ^*$, there exists $\alpha > 0$ such that 
$$\phi_1 = c_1 + O(x^{\alpha}).$$
Note that $\phi_1-c_1 \in x^{2+\epsilon} H_b ^l$ for all $\epsilon < 1/2$ such that $3/2 - \epsilon \neq \mathfrak{R}(\zeta)$ for any $\zeta \in spec_b (L_b)$.  If $\alpha < 2$, since $spec_b (L_b)$ is discrete,  taking $\epsilon \in (3/2 - \alpha, 1/2)$ gives a contradiction.  The same argument applies to $\psi_1$ so
$$\phi_1 - c_1 \textrm{ and } \psi_1 - c_1' = O(x^2) \quad \textrm{as } x \to 0.$$

By hypothesis, 
$$0 \leq \Delta \phi + f(\x, \phi) \implies L \phi \leq F(\x, \phi) = L\phi_1.$$
Since $\phi_1 \to c_1 > c_0$ and $\phi \to c_0$ at $\pa X$, by the standard maximum principle argument, 
\begin{equation} \label{eq:maxp} \phi \leq \phi_1 \implies 0 < \phi_1. \end{equation}

Similarly
$$\Delta \psi + f(\x, \psi) \leq 0 \implies L(\psi) \geq F(\x, \psi) = L(\psi_1) \implies \psi_1 \leq \psi.$$
Since $\phi \leq \psi$, by the increasing property of $F$ and since $\psi_1 \to c_1' > c_1$ at $\pa X$, 
$$L(\phi_1) = F(\x, \phi) \leq F(\x, \psi) = L(\psi_1) \implies \phi_1 \leq \psi_1.$$
Analogous to our construction of $\phi_1$ and $\psi_1$, we may define inductively $c_0 < c_1 < c_2 < \ldots < c_2' < c_1' < c_0'$ and 
\begin{eqnarray} \label{def:induct}  
\phi_0 &=& \phi, \quad \phi_k = L^{-1} (F(\x, \phi_{k-1})), \, k \geq 1 \nonumber; \\
\psi_0 &=& \psi, \quad \psi_k = L^{-1} (F(\x, \psi_{k-1})), \, k \geq 1. \end{eqnarray}
Moreover, we are free to choose the constants such that $c_k \nearrow 1$, $c_k ' \searrow 1$.  By the maximum principle for $L$ and the increasing property of $F(\x,t)$ with respect to $t$, 
$$\phi \leq \phi_1 \leq \phi_2 \ldots \leq \psi_2 \leq \psi_1 \leq \psi.$$
Repeating this argument inductively produces monotonically bounded sequences $\{\phi_k\}$ and $\{\psi_k\} \in \cA_{phg} ^* \cap \cH_b ^{\infty}(X) \subset \cC^{\infty}(X)$.  Therefore, we have pointwise convergence $\phi_k \to \underline{u},$ $\psi_k \to \overline{u}$ with 
$$0 < \phi \leq \underline{u} \leq \overline{u} \leq \psi.$$
By construction, $\phi_k \to c_k$ at $\pa X$ and $\psi_k \to c_k'$ at $\pa X$ and $c_k, c_k' \to 1$ at $\pa X$, so $\underline{u}$, $\overline{u} \to 1$ at $\pa X$.  By the compactness of the parametrix $G$, we may pass to a diagonal subsequence and assume that $\underline{u}$ and $\overline{u} \in H^{l} _b (X)$ for all $l$; the Sobolev embedding theorem applied to a compact exhaustion of $X$ implies $v \in \cC^{\infty}(X)$.  Taking limits in (\ref{def:induct}), 
$$Lv = F(\bx, v) \qquad \textrm{for } v = \underline{u} \textrm{ or } \overline{u}.$$
This is equivalent to 
$$L_b v = x^2 F(\bx, v).$$  
By our preceding arguments using the mapping properties of $L_b$, $G$, and $R$, $v-1 \in x^{2+\epsilon} H^{l} _b (X)$ for all $\epsilon < 1/2$.  By \cite{edge} Theorem 7.14, $v$ admits a partial expansion, and in particular, our previous arguments show that 
$$v \sim 1 + O(x^2) \quad \textrm{ as } x \to 0.$$
This argument can be repeated to produce a full polyhomogeneous expansion at $\pa X$.  Note that we also have 
$$v^{\frac{4}{m-1}} \sim 1 + O(x^2).$$
Therefore, by Proposition \ref{prop:cds}, $v^{\frac{4}{m-1}}g$ is a conic metric.  Since $Lv = F(\bx, v) \iff Pv = 0$, $v^{\frac{4}{m-1}}g$ has constant scalar curvature $-1$.   
\end{proof} 

\section{Examples and further directions}  
\subsection{Negative examples} 
Given a conic metric that satisfies the two conditions in Proposition 1, namely $g$ with 
\[ S(h_0) \equiv c > 0, \quad \sum_{i,j=1} ^m (h_0)^{ij} \frac{ \pa (h_0)_{ij}}{ \pa_x} =0,\] 
does there always exist a conformal deformation to a conic metric with constant scalar curvature, without further assuming $c=m(m-1)$?  If the conformal factor must be polyhomogeneous conformal, it turns out that the answer is \em no.  \em  Consequently, the assumptions of the theorem for such conformal factors are not only necessary but also sufficient as we will demonstrate here.  

One can ask what happens if more
general expansions are allowed?  For example, if the conformal factor 
\[ u \sim u_0 x^{\alpha} + O(x^{\alpha + \eps}),  \quad u_0 \textrm{ and $\eps$ are positive constants.}\]  
Consequently 
\[ u^{\frac{4}{m-1}} \sim u_0^{ \frac{4}{m-1}} x^{\frac{4\alpha}{m-1}} \left( 1 + O(x^\eps) \right), \quad x \to 0.\] 
For the metric 
\[ \tilde{g} := u^{\frac{4}{m-1}} g, \] 

by the calculation (\ref{eq:ug1}), 
\[ S(\tilde{h}_0) = \frac{ \left( \frac{4\alpha}{m-1} + 2\right)^2}{4} S(h_0) = \left( \frac{2\alpha}{m-1} + 1\right)^2 c.\] 
By Proposition 1, to have constant scalar curvature, a necessary condition is that 
\begin{equation} \label{eq:alpha_eq} \left( \frac{2\alpha}{m-1} + 1\right)^2 c = m(m-1) \implies \alpha = \frac{1-m}{2} \pm \frac{m-1}{2} \sqrt{ \frac{m(m-1)}{c}}   \end{equation} 
\[ = \frac{m-1}{2}\left(-1 \pm \sqrt{ \frac{m(m-1)}{c}}\right).\] 
The Yamabe equation for the scalar curvature gives that 
\[ S(\tilde g) = \frac{ - \frac{4m}{m-1} \Delta u + S(g) u}{u^{\frac{m+3}{m-1}}}.\] 
The leading order asymptotic behavior is 
\[ u_0 ^{1-\frac{m+3}{m-1}} \left( - \frac{4 m}{m-1} \left( \alpha (\alpha - 1) +m \alpha\right) + c - m(m-1)\right)x^{\alpha - 2 - \frac{m+3}{m-1}\alpha}, \quad x \to 0.\] 
If the scalar curvature is constant, then the power of $x$ should vanish, or its coefficient should vanish.  The power of $x$ is 
\[ \frac{-4 \alpha}{m-1} - 2 = 0 \iff \alpha = - \frac{m-1}{2}.\] 
This is impossible by \eqref{eq:alpha_eq}.  

The coefficient vanishes if and only if  
\[-\alpha(\alpha+m-1) + c \frac{m-1}{4m} - \frac{(m-1)^2}{4} = 0\]
\[ \iff \alpha^2 + (m-1) \alpha - c \frac{m-1}{4m} + \frac{(m-1)^2}{4} = 0.\] 
That is a quadratic equation for $\alpha$ whose solutions are 
\[ -\frac{(m-1)}{2} \pm \frac 1 2 \sqrt{ (m-1)^2 - 4 \left( -c \frac{m-1}{4m} + \frac{(m-1)^2}{4} \right)}.\] 
\[ = - \frac{(m-1)}{2} \pm  \frac{\sqrt{c(m-1)}}{2 \sqrt m}.\] 
We compare this to the expression for $\alpha$ \eqref{eq:alpha_eq}, which shows that the coefficient vanishes if and only if 
\[  \pm  \frac{\sqrt{c(m-1)}}{2 \sqrt m} = \pm \frac{(m-1)}{2} \sqrt{\frac{m(m-1)}{c}} \iff c = m(m-1) \iff \alpha = 0. \] 
Consequently, for all $\alpha \neq 0$, the scalar curvature of $\tilde g$ either decays or diverges as $x \to 0$ and therefore cannot be constant.  Hence within the class of polyhomogeneous conformal factors, the conditions of Theorem 1 are both necessary and sufficient.


\subsection{Positive examples}
The calculation (\ref{eq:wp}) provides examples of conic metrics which satisfy the hypotheses of the theorem and have non-constant negative scalar curvature.  In particular, fix a closed $m$ dimensional Riemannian manifold $(Y, h)$ with constant scalar curvature $m(m-1)$.  Let $u= u(x)$ be a positive polyhomogeneous conormal function that depends only on $x$ near $\pa X$ such that $u \sim 2 + O(x^{\alpha})$ as $x \to 0$ for some $\alpha \geq 2$.  Assume further that $u$ is at least twice smoothly differentiable, $u' \geq 0$, $u''$ is bounded, and $u \geq 2$ on $X$.  Then, for the metric 

\[ g = dx^2 + x^2 u^2(x) h(\y; d\y) = dx^2 + x^2 H(x, \y; d\y),\] 
by (\ref{eq:wp}), 
\[ S(g) = \frac{1}{x^2 u^2(x)} \left[ S(h)(\y) - 2 m x u(x) (x u''(x) + 2u'(x)) - m(m-1) |x u'(x) + u(x)|^2 \right] \] 
\[ = \frac{1}{x^2 u^2(x)} \left[ m(m-1)(1-|xu'(x) + u(x)|^2) - 2 m x u(x) (xu''(x) + 2 u'(x)) \right].\] 
Since $u \sim 2 + O(x^{\alpha})$ as $x \to 0$ with $\alpha \geq 2$, and $u''(x)$ is bounded, we have 
\[ m(m-1)(1-|xu'(x) + u(x)|^2) - 2 m x u(x) (xu''(x) + 2 u'(x))\] 
\[ \sim -3m(m-1) + \mathcal O (x),  \textrm{ as } x \to 0.\] 
Consequently, there is $\epsilon > 0$ such that for all $x \in (0, \epsilon)$, $S(g)$ is negative.  
By the form of the metric and the definition of $u$, 
\[ \frac{\pa H_{ij} }{\pa x} = 2 u(x) u'(x) h_{ij} \to 0 \textrm{ as } x \to 0.\] 
This shows that condition 2 of Proposition 1 is satisfied. 

Consequently, this construction produces conic metrics with non-constant negative scalar curvature that satisfy the hypotheses of the Theorem.  There are many similar examples.  

\subsection{The orbifold Yamabe invariant}

The Yamabe invariant of orbifolds and cylindrical manifolds is the subject of beautiful work by Akutagawa and Botvinnik \cite{ab1}, \cite{ab2}, and \cite{ab3}.  Since the Stokes formula and divergence theorem hold for conic metrics, Proposition 2.1 of \cite{ab3} extends to our setting.  
Recall the Einstein-Hilbert functional, 
\begin{equation} \label{eq:eh} EH(g) = \frac{\int_M S_g dVol_g}{ \left( \int_M dVol_g \right)^{\frac{m-1}{m+1}}}. \end{equation} 
\begin{prop} Given an incomplete Riemannian manifold $(X,g)$ with conic metric, the set of critical points of $EH$ which are also conic metrics coincides with the set of Einstein conic metrics on $X$.
\end{prop}  
Analogous to the orbifold Yamabe constant of \cite{ab3}, one may define the conic Yamabe constant.  
\begin{defn} \em Let $(X, g)$ be an incomplete Riemannian manifold with conic metric.  If there exists a conformal deformation $u$ such that $\tilde{g} = ug$ and $\tilde{g}$ is also a conic metric, then we write $\tilde{g} \in [g]_{c}$.  We define
$$Y_{[g]} ^{c} (X) := \inf_{\tilde{g} \in [g]_{c}} EH(\tilde{g}).$$
\end{defn} 
Aubin's inequality $-\infty < Y_{[g]} ^{c} (X) \leq Y(\bS^n)$ holds in this context.  The conic Yamabe invariant is defined analogous to the orbifold Yamabe invariant
$$Y^{c} (X) := \sup_{[g]_{c}} Y_{[g]} ^{c} (X),$$
where the supremum is taken over all conformal classes of conic metrics on $X$.  In dimension 4, \cite{ab3} prove results using the modified scalar curvature and techniques of Gursky-LeBrun \cite{g}, \cite{glb}.  We expect many results of \cite{ab3} generalize to conic metrics; these generalizations and conformal deformations to constant positive and zero scalar curvature will be the subject of a future work.  
\section*{Acknowledgments}
We would like to thank the Mathematisches Forschungsinstitut in Oberwolfach, Germany, where the majority of this work was completed in the Research in Pairs Program, and the Mathematical Sciences Research Institute in Berkeley, California, where this work was initiated.  The second author was supported by an NSF-AWM Mentoring Travel Grant and is grateful to Rafe Mazzeo and Pierre Albin for discussions and correspondence.   

\end{document}